\documentclass[reqno,11pt,a4paper]{amsart}

\usepackage{tikz}
\usetikzlibrary{matrix,arrows,calc,snakes,patterns,decorations.markings}

\usepackage[mathscr]{eucal}
\usepackage{graphics,epic}
\usepackage{amsfonts}
\usepackage{amscd}
\usepackage{latexsym}
\usepackage{amsmath,amssymb, amsthm, stmaryrd, bm, bbm}
\usepackage[all,2cell]{xy}
\usepackage{mathrsfs}
\usepackage{url, hyperref}

\hypersetup{
    colorlinks,
    linkcolor={red!50!black},
    citecolor={blue!50!black},
    urlcolor={blue!80!black}
}

\newtheorem{thm}{Theorem}[section]
\newtheorem*{thm*}{Theorem}

\newtheorem{lem}[thm]{Lemma}
\newtheorem{prop}[thm]{Proposition}
\newtheorem{cor}[thm]{Corollary}

\newtheorem*{conj*}{Conjecture}

\newtheorem*{question*}{Question}
\theoremstyle{remark}
\newtheorem{rem}[thm]{Remark}

\theoremstyle{definition}
\newtheorem{defn}[thm]{Definition}

\newcommand{\ul}[1]{\underline{#1}}

\newcommand{\opname}[1]{\operatorname{\mathsf{#1}}}

\renewcommand{\mod}{\opname{mod}\nolimits}
\newcommand{\ssmod}{\opname{ssmod}\nolimits}

\newcommand{\grmod}{\opname{grmod}\nolimits}
\newcommand{\proj}{\opname{proj}\nolimits}

\newcommand{\add}{\opname{add}\nolimits}

\newcommand{\rad}{\opname{rad}\nolimits}

\newcommand{\grproj}{\opname{grproj}}

\newcommand{\Z}{\mathbb{Z}}
\newcommand{\N}{\mathbb{N}}

\newcommand{\C}{\mathbb{C}}

\newcommand{\id}{\mathrm{id}}

\newcommand{\Hom}{\opname{Hom}}
\newcommand{\End}{\opname{End}}

\newcommand{\Ext}{\opname{Ext}}

\newcommand{\GL}{\opname{GL}}

\newcommand{\ca}{{\mathcal A}}

\newcommand{\ch}{{\mathcal H}}

\newcommand{\cl}{{\mathcal L}}

\newcommand{\co}{{\mathcal O}}

\newcommand{\ct}{{\mathcal T}}

\renewcommand{\hat}[1]{\widehat{#1}}

\newcommand{\mm}[1]{#1}

\numberwithin{equation}{section}

\begin{document}

\title{A new equivalence between singularity categories of commutative algebras}

\author{Martin Kalck}
\email{martin.maths@posteo.de}

\begin{abstract}
We construct a triangle equivalence between the singularity categor\mm{ies of two isolated cyclic quotient singularities of Krull dimensions two and three, respectively.} 
This is the first example of a singular equivalence involving connected commutative algebras of odd and even Krull dimension.
\mm{In combination with Orlov's localization result, this gives further singular equivalences between certain quasi-projective varieties of dimensions two and three, respectively.}
\end{abstract}

\maketitle

\thispagestyle{empty}

\section{Introduction}
\noindent
Singularity categories were introduced by Buchweitz providing a general framework for Tate cohomology \cite{Buchweitz}. More recently, they have been studied in birational geometry \cite{Wemyss18}, in relation with string theory \& homological mirror symmetry \cite{Orlov2} and knot theory \cite{KR}.

There are many known triangle equivalences between singularity categories $D_{sg}(A):=D^b(A)/K^b(\proj-A)$ of (\emph{noncommutative}) connected Noetherian $\C$-algebras $A$. For example, for every ADE-surface singularity $R$ (except for $E_8$), we construct \mm{\emph{uncountably}} many algebras $A_i$, such that
$\mod A_i \ncong \mod A_j$ for $i \neq j$ and $D_{sg}(A_i) \cong D_{sg}(R)$ for all $i$, see \cite{KIWY15}.

This is in stark contrast to the commutative case: until recently, Kn{\"o}rrer's equivalences \begin{align}\label{E:Knoerrer}
D_{sg}(S/(f)) \cong D_{sg}(S \llbracket x,y \rrbracket/(f-xy)), \text{ for non-zero } f \in S:=\mathbb{C} \llbracket z_0, \ldots, z_d \rrbracket,
\end{align}
from 1987  \cite{KnorrerCMmodules} were the \emph{only} \cite{Matsui} known source of singular equivalences between singular connected commutative Noetherian $\C$-algebras, which are not analytically isomorphic.

Using the relative singularity category techniques developed in our joint works \cite{KY16, KY18}, one can construct the following additional family of singular equivalences \begin{align}\label{E:Yang}
D_{sg}\left(\C\llbracket y_1, y_2\rrbracket^{\frac{1}{n}(1, 1)}\right) \cong D_{sg}\left(\frac{\C[z_1, \ldots, z_{n-1}]}{(z_1, \ldots, z_{n-1})^2}\right),
\end{align}
where for a primitive $n$th root of unity $\epsilon \in \C$, we use the notation: 
\begin{align}
{\frac{1}{n}(a_1, \ldots, a_m)}=\left\langle \mathsf{diag}\left(\epsilon^{a_1}, \ldots \epsilon^{a_m}\right)
\right\rangle \subset \GL(\mm{m}, \C).
\end{align}
We learned about \eqref{E:Yang} from \cite{YangPrivate}. The equivalences \eqref{E:Yang} can also be deduced from \cite{Kawamata15} in combination with \cite{OrlovIdempotent}.
In our joint work with Joseph Karmazyn \cite{KK17}, we give another geometric proof of \eqref{E:Yang} \mm{and construct noncommutative finite dimensional algebras $K_{n, a}$} that generalize \mm{\eqref{E:Yang}} to all cyclic quotient surface singularities \mm{$\C\llbracket y_1, y_2\rrbracket^{\frac{1}{n}(1, a)}$}.

We state the main result of this paper. To the best of our knowledge, this is the first example of a singular equivalence between rings of odd and even Krull dimension.

\begin{thm}\label{T:Main}
There is a triangle equivalence between singularity categories
\begin{align}\label{E:first}
D_{sg}\left(\C\llbracket x_1, x_2, x_3\rrbracket^{\frac{1}{2}(1, 1, 1)}\right) \cong D_{sg}\left(\C\llbracket y_1, y_2\rrbracket^{\frac{1}{4}(1, 1)}\right). \end{align}
Moreover, these singularity categories are equivalent to the following singularity category:
\begin{align}\label{E:RadS}
 D_{sg}\left(\frac{\C[z_1, z_2, z_3]}{(z_1, z_2, z_3)^2}\right). 
 \end{align}
\end{thm}
\begin{rem}
(1) The category \eqref{E:RadS} is idempotent complete, Hom-infinite, not Krull--Schmidt and can be explicitly described using  the Leavitt algebra of type $(1, 2)$ \cite{ChenYang, Leavitt}.

\noindent
(2) By \cite[Ex 2.15 \& 2.17]{PavicShinder}, there are isomorphisms involving Grothendieck groups  
\begin{align}\label{E:Grothen}
K_0\left(D_{sg}\left(\C\llbracket x_1, x_2, x_3\rrbracket^{\frac{1}{2}(1, 1, 1)}\right)\right) \cong \Z/4\Z \cong K_0\left(D_{sg}\left(\C\llbracket y_1, y_2\rrbracket^{\frac{1}{4}(1, 1)}\right)\right),
\end{align}
which provide evidence for \eqref{E:first} and will be used in our proof below.

\mm{
\noindent
(3) Let $X$ and $Y$ be complex quasi-projective varieties of Krull dimensions $2$ and $3$, respectively, with only isolated singularities $\mathsf{Sing}(X)=\{s_1, \ldots, s_n\}$ and $\mathsf{Sing}(Y)=\{t_1, \ldots, t_n\}$.

Assume that the $\mathfrak{m}$-adic completions of the local rings in the singular points satisfy:
\begin{align}
\hat\co_{X, {s_i}} \cong \C\llbracket y_1, y_2\rrbracket^{\frac{1}{4}(1, 1)}   \text{ and } \hat\co_{Y, t_i} \cong \C\llbracket x_1, x_2, x_3\rrbracket^{\frac{1}{2}(1, 1, 1)}.
\end{align}
There are triangle equivalences, where $(-)^\omega$ denotes the idempotent completion, cf. \cite{OrlovIdempotent}.
\begin{align}\label{E:SingEqQuasiProj}
D_{sg}(X)^\omega \cong \bigoplus_{i=1}^n D_{sg}\left(\C\llbracket y_1, y_2\rrbracket^{\frac{1}{4}(1, 1)}\right) \cong \bigoplus_{i=1}^n D_{sg}\left(\C\llbracket x_1, x_2, x_3\rrbracket^{\frac{1}{2}(1, 1, 1)}\right) \cong  D_{sg}(Y)^\omega. 
\end{align}
The equivalence in the middle is \eqref{E:first} and the equivalences on the left and on the right follow from \cite{OrlovIdempotent}, since the categories in Theorem \ref{T:Main} are idempotent complete, cf. part (1) of this remark. We can rephrase \eqref{E:SingEqQuasiProj} as follows: the singularity categories $D_{sg}(X)$ and $D_{sg}(Y)$ are triangle equivalent up to taking direct summands.
}
\end{rem}

\subsection{Strategy of proof} 
By \eqref{E:Yang}, it suffices to show that there is a singular equivalence between $R_{1, 1, 1}:=\C\llbracket x_1, x_2, x_3\rrbracket^{\frac{1}{2}(1, 1, 1)}$ and $K_{4,1}:=\C[z_1, z_2, z_3]/(z_1, z_2, z_3)^2$.
Building on work of Auslander \& Reiten \cite{AuslanderReiten} and Keller \& Vossieck \cite{KellerVossieck87}, we show that $D_{sg}(R_{1, 1, 1})$ is triangle equivalent to the Heller stabilization \cite{Heller} of a left triangulated category, whose underlying \mm{$\C$-linear} category is semisimple abelian with a unique simple object -- this object corresponds to the maximal Cohen--Macaulay \mm{$R_{1, 1, 1}$-}module $M=\Omega(\omega_{R_{1, 1, 1}})$. By Proposition \ref{P:Key}, the structure of triangulated categories $\ct$ arising in this way, is completely determined by the order of their Grothendieck groups $K_0(\ct)$ -- indeed, if $\left|{K_0(\ct)}\right|=n$ then $\ct \cong D_{sg}(K_{n,1})$, where $K_{n, 1}=\C[z_1, \ldots, z_{n-1}]/(z_1, \ldots, z_{n-1})^2$. This result is inspired by Chen \cite{ChenQuadratic}. Combining this with \eqref{E:Grothen}, shows Theorem \ref{T:Main}.

We will prove the theorem in Section \ref{S:Proof} using the preparations in Section \ref{S:Prep}.

\medskip
\noindent
\emph{Acknowledgement}.
I am very grateful to Evgeny Shinder for a remark that started this work. Further inspiration came from my joint works with Dong Yang \cite{KY16, KY18} \& Joe Karmazyn \cite{KK17} and from email exchanges with Xiao-Wu Chen and Yujiro Kawamata. I would like to thank Xiao-Wu Chen, Yujiro Kawamata, Michael Wemyss and Dong Yang for comments on this text and for answering my questions. \mm{I am very grateful to the referee for detailed comments, suggestions and corrections, which improved this text. In particular, the referee pointed out the examples of \emph{graded} singular equivalences \cite[Corollary 3.23 (d)]{HIMO}, which led to the new appendix. }

\section{Preparation: stabilization of left triangulated categories}\label{S:Prep}
\noindent
\mm{We fix a field $k$}. All categories and functors below are $k$-linear. \mm{For a right Noetherian $k$-algebra $R$, we write $\mod R$ for the category of finitely generated right $R$-modules.} 

Our main reference for Heller's stabilization \cite{Heller} of 
looped and left triangulated categories is \cite{ChenQuadratic}, which builds on works of Keller \& Vossieck \cite{KellerVossieck87} and Beligiannis \cite{Bel00}.

\begin{lem}\label{L:Unique}
Let $\ch$ be a semisimple abelian category, with an endofunctor $\Omega\colon \ch \to \ch$.

Then the only left triangulated structure on $(\ch, \Omega)$ is the trivial structure -- i.e. all left triangles are isomorphic to direct sums of trivial left triangles. 
\end{lem}
\begin{proof}
\mm{By an axiom of left triangulated categories $\cl$ (cf. \cite{BM}), every morphism $B \xrightarrow{f} C$ in $\cl$ gives rise to a left triangle 
\begin{align}\label{E:leftTriangle}
\Omega(C) \to A \to B \xrightarrow{f} C. 
\end{align}
One can check that every left triangle in $\cl$ is isomorphic to a left triangle of the form \eqref{E:leftTriangle}, by adapting the proof for triangulated categories (cf.~e.g.~\cite[Prop.~10.1.15]{KashiwaraSchapira}) to the setup of left triangulated categories}. Since $\ch$ is semisimple, $f$ is isomorphic to 
\begin{align}
f'\colon B' \oplus D \xrightarrow{\begin{pmatrix} 0 & 0 \\ 0 & \id_D \end{pmatrix}} C' \oplus D. 
\end{align}
Completing $f'$ to a left triangle yields a direct sum of trivial left triangles, which is isomorphic to \eqref{E:leftTriangle}.
\end{proof}

\begin{cor}\label{C:sesisubcat}
Let $\ch$ be a full subcategory of a left triangulated category $(\cl, \Omega)$.
Assume that $\ch$ is semisimple abelian and $\Omega(\ch) \subseteq \ch$.
 
Then $(\ch, \Omega)$ is a left triangulated subcategory of $(\cl, \Omega)$.
\end{cor}
\begin{proof}
Following the proof of Lemma \ref{L:Unique}, shows that taking all left triangles in $(\cl, \Omega)$, which lie in $\ch$ induces the trivial left triangulated structure on $(\ch, \Omega)$. By Lemma \ref{L:Unique} this is the unique left triangulated structure on $(\ch, \Omega)$.   
\end{proof}

\noindent
A pair of a category $L$ with an endofunctor $\Omega$ is called a \emph{looped category}. A functor $F$ between looped categories $(L_1, \Omega_1), (L_2, \Omega_2)$, such that there is a natural isomorphism $F\Omega_1\cong \Omega_2 F$, is called a \emph{looped functor}.
For any looped category $(L, \Omega)$, Heller constructs its \emph{stabilization} $(S(L, \Omega), \Sigma)$, a looped category with an automorphism $\Sigma$ and a looped functor
$\mathsf{S} \colon L \to S(L, \Omega)$. The pair $(S(L, \Omega), \Sigma), \mathsf{S}$ enjoys a universal property \cite[Proposition 1.1]{Heller}. Using this property, one can show that a looped functor $F$ between looped categories induces a functor $S(F)$ between their stabilizations. The functor $S(F)$ is called \mm{the} \emph{stabilization} of $F$. Finally, if $L$ is left triangulated, then
$S(L, \Omega)$ is a triangulated category with shift functor $\Sigma^{-1}$ and triangles in $S(L, \Omega)$ are induced by left triangles in $L$, cf.~\cite[Section 3]{Bel00}.

\begin{cor}\label{C:Stabili}
We keep the assumptions in Corollary \ref{C:sesisubcat}. In addition, we assume that for any $X \in \cl$, there exists $n(X) \in \N$, such that
$\Omega^{n(X)} \in \ch$. Then the stabilization of the inclusion $(\ch, \Omega) \subseteq (\cl, \Omega)$ yields a triangle equivalence 
$S(\ch, \Omega) \cong S(\cl, \Omega)$.
\end{cor}
\begin{proof}
This follows from Corollary \ref{C:sesisubcat} in combination with \cite[Cor 2.3. \& 2.7 ]{ChenQuadratic}.
\end{proof}

\noindent
Let $R$ be a \mm{right} Noetherian ring and let $\Omega$ be the syzygy functor on the stable module category $\ul{\mod}\,{R}:=\mod R/\proj R$. Then the pair $(\ul{\mod}\,{R}, \Omega)$ has a left triangulated structure \cite[Theorem 3.1.]{BM}, where the left triangles are isomorphic to sequences 
$\Omega X \xrightarrow{u} Y \xrightarrow{v} Z \xrightarrow{w} X$ arising from commutative diagrams in $\mod R$, where $P$ is projective and the rows are exact:
\[
\xymatrix{0 \ar[r] & \Omega X \ar[r] \ar[d]^u& P \ar[r] \ar[d] & X \ar[r] \ar[d]^{\id_X}& 0 \\
0 \ar[r] &Y \ar[r]^v & Z \ar[r]^w & X \ar[r] & 0 \mm{.}
}
\]

\noindent
The next result generalizes Buchweitz\mm{'s} equivalence \cite{Buchweitz} from Gorenstein rings to Noetherian rings, see \cite{KellerVossieck87} \& \cite[Cor.~3.9(1)]{Bel00}. It is a key ingredient in our proof of Theorem \ref{T:Main}. 

\begin{thm}\label{T:KV}
Let $R$ be a \mm{right} Noetherian ring. Then there is a triangle equivalence
\begin{align}
S(\ul{\mod}\,{R}, \Omega) \cong D_{sg}(R).
\end{align}
\end{thm}

\mm{
\begin{defn}
We call a $k$-linear category $\ca$ \emph{split simple}, if there is a $k$-linear equivalence \begin{align} \ca \cong \mod k.
\end{align}
\end{defn}
\begin{rem}
If $k$ is algebraically closed and $\ca$ is Hom-finite and semisimple abelian with a unique simple object (up to isomorphism), then $\ca$ is split simple.
\end{rem}}

Following \cite{KK17}, we denote the algebra $k[z_1, \ldots, z_{n-1}]/(z_1,\ldots, z_{n-1})^2$ by $K_{n, 1}$. 
\mm{One of our main examples of a split simple category is the subcategory of semisimple $K_{n, 1}$-modules, which we denote by $\ssmod K_{n, 1}$. Its objects are isomorphic to finite direct sums of the simple $K_{n, 1}$-module $k[z_1, \ldots, z_{n-1}]/(z_1,\ldots, z_{n-1})$. We will write $\ul{\ssmod} \ K_{n, 1}$ for the image of $\ssmod K_{n, 1}$ under the additive quotient functor
\begin{align}
q \colon \mod K_{n, 1} \to \ul{\mod} \ K_{n,1}:=\mod K_{n, 1}/\proj K_{n, 1}.
\end{align}
Since the semisimple $K_{n, 1}$-modules are not projective, $q$ induces a $k$-linear equivalence $\ssmod K_{n, 1} \cong \ul{\ssmod} \ K_{n, 1}$. The syzygy $\Omega_{K_{n, 1}}(M)$ of a finitely generated $K_{n, 1}$-module $M$ is a submodule of the radical $\rad (K_{n, 1}^{\oplus c}) \cong (\rad K_{n, 1})^{\oplus c}$, which is semisimple. Thus, 
\begin{align} \label{E:SysSeSi}
\Omega_{K_{n, 1}}(M) \in \ul{\ssmod} \ K_{n, 1}.
\end{align}  
It follows from Corollary \ref{C:sesisubcat} that 
\begin{align}\label{2.6}
(\ul{\ssmod} \ K_{n, 1}, \Omega_{K_{n, 1}}) \subseteq (\ul{\mod} \ K_{n, 1}, \Omega_{K_{n, 1}})  
\end{align}
is a left triangulated subcategory.}

We have the following classification of stabilizations of \mm{split simple} categories.

\begin{prop}\label{P:Key} 
Let $(\ch, \Omega)$  be a left triangulated category, 
where $\ch$ is a \mm{split simple} category with simple object $s$. 
Then the following statements hold:
\begin{enumerate}
\item $\Omega(s) \cong s^{\oplus (n-1)}$ in $\ch$, for some integer $n \geq 1$.
\item There is a triangle equivalence $S(\ch, \Omega) \cong D_{sg}\left(K_{n, 1}\right)$.
\item There is an isomorphism of groups $K_0(S(\ch, \Omega)) \cong \Z/n\Z$.
\end{enumerate}
\end{prop}
\begin{proof}
Since $\ch$ is \mm{split simple}, all objects in $\ch$ are isomorphic to $s^{\oplus m}$ for some $m \geq 0$. This shows (1). \mm{Let us show} (2). \mm{By \eqref{E:SysSeSi},} the inclusion of left triangulated categories \mm{\eqref{2.6} satisfies the assumptions of Corollary \ref{C:Stabili}. This} gives a triangle equivalence between the stabilizations, cf. \cite[Proof of Thm 4.4.]{ChenQuadratic}.
\begin{align}
S(\ul{\ssmod}\,K_{n, 1}, \Omega_{K_{n, 1}}) \cong S(\ul{\mod}\,K_{n, 1}, \Omega_{K_{n, 1}}).
\end{align}
By Theorem \ref{T:KV}, the stabilization of $(\ul{\mod}\,K_{n, 1}, \Omega_{K_{n, 1}})$ is triangle equivalent to $D_{sg}(K_{n,1})$. Using Lemma \ref{L:Unique} and (1), the left triangulated category $(\ul{\ssmod}\,K_{n, 1}, \Omega_{K_{n, 1}})$ is equivalent to $(\ch, \Omega)$. In particular, their stabilizations are triangle equivalent. Summing up, we have a chain of triangle equivalences proving (2):
\begin{align}
S(\ch, \Omega) \cong S(\ul{\ssmod}\,K_{n, 1}, \Omega_{K_{n, 1}}) \cong S(\ul{\mod}\,K_{n, 1}, \Omega_{K_{n, 1}}) \cong D_{sg}\left(K_{n, 1}\right)\mm{.}
\end{align}
Finally, (3) follows from (2)  and 
$K_0(D_{sg}(A)) \cong \Z/\dim(A) \Z$ for finite dimensional local algebras $A$. 
\end{proof}

The following structure result is a direct consequence. This will not be used later. 

\begin{cor}
For $i=1, 2$, let $(\ch_i, \Omega_i)$  be left triangulated categories, 
where the $\ch_i$ are \mm{split simple} categories. Then the following statements are equivalent.
\begin{enumerate}
\item There is a triangle equivalence between the stabilizations $S(\ch_1, \Omega_1) \cong S(\ch_2, \Omega_2)$.
\item There is an isomorphism of Grothendieck groups $K_0(S(\ch_1, \Omega_1)) \cong K_0(S(\ch_2, \Omega_2))$.
\end{enumerate}
\end{cor}

\begin{cor}\label{C:Servier}
Let $R$ be a Noetherian $k$-algebra with syzygy functor $\Omega_R\colon \ul{\mod} \ R \to \ul{\mod} \ R$. 

\mm{Assume there is a finitely generated $R$-module $M$ such that the following conditions hold:
\begin{itemize}
\item[(s1)] $\ul{\End}_R(M) \cong k$.
\item[(s2)] $\Omega_R(M) \cong M^{\oplus n-1}$ for an integer $n \in \Z_{>0}$.
\item[(s3)] For every finitely generated $R$-module $N$ there is an integer $d \in \Z_{\geq 0}$ such that 
$\Omega_R^d(N) \cong M^{\oplus m}$. 
\end{itemize}
} 

\noindent
Then \mm{there is an isomorphism of groups}
\begin{align}
K_0(D_{sg}(R)) \cong \Z/n\Z
\end{align}
and there is a triangle equivalence
\begin{align}
D_{sg}(R) \cong D_{sg}(k[z_1, \ldots, z_{n-1}]/(z_1,\ldots, z_{n-1})^2)\mm{.}
\end{align}
\end{cor}
\begin{proof}
\mm{Consider the additive subcategory $\ch:=\ul{\add} \ M \subseteq \ul{\mod} \ R$ generated by $M$. Condition (s1) implies that $\ch$ is split simple with simple object $M$ and condition (s2) ensures that the assumptions of Corollary \ref{C:sesisubcat} are satisfied. Using (s3), we can apply}
Corollary \ref{C:Stabili} \mm{to get} a triangle equivalence
$S(\ch, \Omega_R) \cong S(\ul{\mod}\,R, \Omega_R)$.
By Theorem \ref{T:KV}, we have a triangle equivalence
$S(\ul{\mod}\,R, \Omega_R) \cong D_{sg}(R)$. \mm{Applying} Proposition \ref{P:Key} \mm{to $(\ch, \Omega_R)$ completes} the proof.
\end{proof}

\section{Proof of the Theorem} \label{S:Proof}
\noindent
First, we recall the following special case of the singular equivalence \eqref{E:Yang}
\begin{align}
D_{sg}\left(\C\llbracket y_1, y_2\rrbracket^{\frac{1}{4}(1, 1)}\right) \cong D_{sg}\left(\frac{\C[z_1, z_2 , z_{3}]}{(z_1, z_2, z_{3})^2}\right).
\end{align}
\mm{In order to prove Theorem \ref{T:Main}, it remains to prove the following} equivalence
\begin{align}
\mm{D_{sg}(R_{1, 1, 1}):=}D_{sg}\left(\C\llbracket x_1, x_2, x_3\rrbracket^{\frac{1}{2}(1, 1, 1)}\right) \cong D_{sg}\left(\frac{\C[z_1, z_2, z_3]}{(z_1, z_2, z_3)^2}\right),
\end{align}
\mm{which follows from Corollary \ref{C:Servier}, using Proposition \ref{P:CheckSyzygySimple} and $K_0(D_{sg}(R_{1, 1, 1})) \cong \Z/4\Z$ (cf.~\cite[Ex 2.17]{PavicShinder}) to check that the assumptions of Corollary \ref{C:Servier} are satisfied for $n=4$.

The proof of Proposition \ref{P:CheckSyzygySimple} uses some facts 
about the stable category of maximal Cohen--Macaulay $R_{1, 1, 1}$-modules $\ul{\mathsf{CM}}\, R_{1, 1, 1}:=\mathsf{CM}\, R_{1, 1, 1}/\proj R_{1, 1, 1}$. We have collected this information in the following proposition, cf.~\cite{AuslanderReiten} \& \cite{Yoshino}.   

\begin{prop} \label{P:Yoshino}
\begin{itemize}
\item[(a)] There is a $\C$-linear equivalence $\ul{\mathsf{CM}}\, R_{1, 1, 1} \cong \ul{\add} \ \omega \oplus M$, where $\omega$ denotes the canonical $R_{1, 1, 1}$-module and $M:=\Omega(\omega)$ is its first syzygy. Both $\omega$ and $M$ are indecomposable.
\item[(b)] $\ul{\End}_R(M) \cong \C$.
\end{itemize}  
\end{prop}
}
\begin{proof}
\mm{We first describe the Auslander--Reiten quiver of $\mathsf{CM}\, R_{1, 1, 1}$ using \cite[Prop.~16.10]{Yoshino} and then explain how both statements (a) and (b) can be deduced from this:}
\begin{equation}\label{E:ARQ}\begin{tikzpicture}[description/.style={fill=white,inner sep=2pt}, baseline=($(current bounding  box.north) + (current bounding  box.center)$)/2]

\matrix (n) [matrix of math nodes, row sep=2em,
                 column sep=2em, text height=1.5ex, text depth=0.25ex,
                 inner sep=2pt, nodes={inner xsep=0.3333em, inner
ysep=0.3333em}] at (0, 0)
    {    \omega  && M && R_{1, 1, 1}  &&& \omega  && M\\
          };
          
\node at ($(n-1-1.west) - (7mm,0mm)$) {(i)};   
\node at ($(n-1-8.west) - (7mm,0mm)$) {(ii)};      

\draw[->] ($(n-1-8.east) + (0,1mm)$) to ($(n-1-10.west) + (0mm,1mm)$);
    \draw[->] ($(n-1-8.east) + (0,0mm)$) to ($(n-1-10.west) + (0mm,0mm)$);
    \draw[->] ($(n-1-8.east) + (0,-1mm)$) to ($(n-1-10.west) + (0mm,-1mm)$);


    \draw[->] ($(n-1-1.east) + (0,1mm)$) to ($(n-1-3.west) + (0mm,1mm)$);
    \draw[->] ($(n-1-1.east) + (0,0mm)$) to ($(n-1-3.west) + (0mm,0mm)$);
    \draw[->] ($(n-1-1.east) + (0,-1mm)$) to ($(n-1-3.west) + (0mm,-1mm)$);
    
 \draw[<-] ($(n-1-5.west) + (0,1mm)$) to ($(n-1-3.east) + (0mm,1mm)$);
    \draw[<-] ($(n-1-5.west) + (0,0mm)$) to ($(n-1-3.east) + (0mm,0mm)$);
    \draw[<-] ($(n-1-5.west) + (0,-1mm)$) to ($(n-1-3.east) + (0mm,-1mm)$);
    
      \draw[->] ($(n-1-5.south west) + (1mm,-0mm)$) .. controls +(-3.7mm,-4mm)
and +(+3.7mm,-4mm) .. node[yshift=-3mm] [midway]{} ($(n-1-1.south east) + (0mm,0mm)$);

      \draw[->] ($(n-1-5.south west) + (2mm,-0mm)$) .. controls +(-3.7mm,-6mm)
and +(+3.7mm,-6mm) .. node[yshift=-3mm] [midway]{} ($(n-1-1.south east) + (-2mm,0mm)$);

      \draw[->] ($(n-1-5.south west) + (3mm,-0mm)$) .. controls +(-3.7mm,-8mm)
and +(+3.7mm,-8mm) .. ($(n-1-1.south east) + (-4mm,0mm)$);
\end{tikzpicture}\end{equation}
\mm{\noindent
The quiver \eqref{E:ARQ} (i) is the Auslander--Reiten quiver in \cite[(16.10.5)]{Yoshino}\footnote{There are two small typos in \cite[(16.10.5)]{Yoshino}: both $R$ and $S_{-1}$ should be replaced by $\widehat{R}$ and $\widehat{S_{-1}}$, respectively.} translated to our notation: firstly, $R_{1, 1, 1} \cong \widehat{R}(^t(2, -1, -1, -1))=\widehat{R}$, by the comment between \cite[Prop.~16.10]{Yoshino} and its proof. Moreover, $\omega \cong \widehat{S_{-1}}$ is the canonical module by \cite[top of p.~150]{Yoshino} and $M \cong \Omega(\omega)$ follows from \cite[(16.10.3)]{Yoshino}.

The quiver \eqref{E:ARQ} (ii), describes indecomposable objects and irreducible morphisms in the stable category $\ul{\mathsf{CM}}\, R_{1, 1, 1}$.
It is obtained from \eqref{E:ARQ} (i) by removing the vertex $R_{1, 1, 1 }$ and all arrows incident to it. 
In particular, $\omega$ and $M$ form a complete set of isomorphism classes of indecomposable objects in the stable category. This shows (a). 
There are no arrows starting in the vertex $M$ of the quiver \eqref{E:ARQ} (ii).
This implies that the only non-trivial stable endomorphisms of $M$ are of the form $\C \cdot \id_M$ and completes the proof.  }
\end{proof}

\mm{
By the discussion above, the following proposition completes the proof of Theorem \ref{T:Main}.
\begin{prop} \label{P:CheckSyzygySimple}
Let $\omega \in \mathsf{CM} R_{1, 1, 1}$ be the canonical module. Then $M:=\Omega(\omega)$ satisfies the conditions (s1) -- (s3) of Corollary \ref{C:Servier} for some integer $n \in \Z_{>0}$.
\end{prop}
\begin{proof}
Condition (s1) holds by Proposition \ref{P:Yoshino} (b). To see (s2), we first note that the syzygy $\Omega(N)$ of a maximal Cohen--Macaulay module $N$ is again maximal Cohen--Macaulay, \cite[Prop. 1.3.]{Yoshino}. In combination with 
Proposition \ref{P:Yoshino} (a), this shows that
\begin{align}\label{E:s2}
\Omega(M) \cong M^{\oplus n-1} \oplus \omega^{\oplus k} \qquad \text{ for } n \in \Z_{>0} \text{ and } k \in \Z_{\geq 0}.
\end{align}
If $k>0$, then
$0 \neq \ul{\Hom}_{R_{1,1,1}}(\Omega(M), \omega) \cong \Ext^1_{R_{1,1,1}}(M, \omega) $, contradicting the fact that $\omega$ is an injective object in $\mathsf{CM}\,{R_{1,1,1}}$, cf. \cite[Cor. 1.13]{Yoshino}. This shows that $k=0$ and therefore \eqref{E:s2} implies (s2). We prove that condition (s3) holds by showing
\begin{align}\label{E:4S}
\Omega^4(X) \in \ul{\add}\, M,
\end{align}
for all $X \in \mod\,{R_{1,1,1}}$. Indeed, since $R_{1,1,1}$ is Cohen--Macaulay of Krull dimension $3$, we see that $\Omega^3(X)$ is a maximal Cohen--Macaulay module, cf.~\cite[Prop.~1.4]{Yoshino}. Thus Proposition \ref{P:Yoshino} (a) shows that $\Omega^3(X) \cong M^{\oplus m} \oplus \omega^{\oplus l}$. Now $M=\Omega(\omega)$ and condition (s2) show \eqref{E:4S}.
\end{proof}
}

\mm{ 
\section{Appendix: Graded versus ungraded singular equivalences}
\noindent
For group graded algebras $A$, one can define \emph{graded singularity categories} 
\begin{align}
D_{sg}^{gr}(A)=D^b(\grmod A)/K^b(\grproj A).
\end{align}
In this appendix, we describe graded commutative algebras $R_{d+1}$ and $R_{d+2}$ of Krull dimensions $d+1$ and $d+2$, respectively, which are known \cite[Corollary 3.23 (d)]{HIMO} to have equivalent graded singularity categories.
Then we show that their `ungraded' singularity categories are not triangle equivalent, by proving the following more general statement: 
\begin{align}\label{E:Knoerrer1}
D_{sg}(S/(f)) \not\cong D_{sg}(S \llbracket X_0 \rrbracket/(f-X_0^2)), \text{ for non-zero } f \in S:=\mathbb{C} \llbracket X_1, \ldots, X_{d +1} \rrbracket,
\end{align}
such that $(S/(f))_\mathfrak{p}$ is non-singular for all prime ideals $\mathfrak{p} \neq (X_1, \ldots, X_{d +1})$. In other words, 
for isolated singularities, Kn{\"o}rrer's equivalences \cite{KnorrerCMmodules} are not $1$-periodic. This answers a question of E. Shinder.

We first describe the algebras $R_{d+1}\cong R$ and $R_{d+2} \cong R'$ appearing in \cite[Corollary 3.23 (d)]{HIMO}. Note that our presentation \eqref{E:present} for $R_{d+2}$ is obtained from the `normal form' \cite[top of p.13]{HIMO} by renaming the variables as follows: $X_i \to X_{i-1}$.  Let $P=\C[X_1, \ldots, X_{d+2}]$, then 
\begin{align}\label{E:present}
R_{d+1}=P/(f)  \quad \text{   and   } \quad R_{d+2}=P[X_{0}]/(f - X^2_{0}), \text{  where   }
\end{align}
\begin{align}
f=X_{d+2}^{p_{d+2}} - \cdots - X_1^{p_1} \qquad \text{  for  } (p_1, \ldots, p_{d+2}) \in \Z_{\geq 2}
\end{align}
and the grading groups $\mathbb{L}_{d+1}$ and $\mathbb{L}_{d+2}$ are abelian of rank $1$, possibly with torsion elements.


Assume there is a singular equivalence \eqref{E:HIMO}. We show that this yields a contradiction.
\begin{align}\label{E:HIMO}
D_{sg}(R_{d+1}) \cong D_{sg}(R_{d+2}).
\end{align}
The algebras $R_{d+1}$ and $R_{d+2}$ have unique isolated singularities in $\mathfrak{m}_{d+1}=(X_1, \ldots, X_{d+2})$ and $\mathfrak{m}_{d+2}=(X_0, \ldots, X_{d+2})$, respectively, cf. \cite[Prop. 2.36]{HIMO}. Therefore, passing to the idempotent completions of \eqref{E:HIMO} yields  
singular equivalences between $\mathfrak{m}_i$-adic completions
\begin{align}\label{E:HIMOcompl}
D_{sg}(\widehat{R}_{d+1}) \cong D_{sg}(\widehat{R}_{d+2}),
\end{align}
cf. \cite[Theorem 3.2(b)]{IyamaWemyss}, which follows from \cite{OrlovIdempotent}. The algebras $\widehat{R}_{d+1}$ and $\widehat{R}_{d+2}$ are special cases of the algebras appearing in \eqref{E:Knoerrer1}. In particular, \eqref{E:HIMOcompl}  contradicts \eqref{E:Knoerrer1}, showing that there are no singular equivalences \eqref{E:HIMOcompl} or \eqref{E:HIMO}.

We now prove \eqref{E:Knoerrer1} by assuming that there is a singular equivalence 
\begin{align}\label{E:Knoerrer2}
D_{sg}(S/(f)) \cong D_{sg}(S \llbracket X_0 \rrbracket/(f-X_0^2)), \text{ for a non-zero } f \in S:=\mathbb{C} \llbracket X_1, \ldots, X_{d +1} \rrbracket,
\end{align}
as in \eqref{E:Knoerrer1} and show that this yields a contradiction. Let $R=S/(f)$ and $R_1=S \llbracket X_0 \rrbracket/(f-X_0^2)$. By assumption, $R$ has isolated singularities and is Gorenstein of Krull dimension $d$. Therefore, $D_{sg}(R)$ is Hom-finite and has Serre functor $\mathbb{S}=[d-1]$, by \cite{Aus84} and  \cite{Aus78}, respectively. The equivalence \eqref{E:Knoerrer2} implies that $D_{sg}(R_1)$ is also Hom-finite and thus $R_1$ also has isolated singularities, see \cite{Aus84}. By applying \cite{Aus78} to the $(d+1)$-dimensional Gorenstein algebra $R_1$, we see that the Serre functor of $D_{sg}(R_1)$ is $\mathbb{S}_1=[d]$. Since Serre functors are unique up to isomorphism \cite{BondalKapranov}, we see that an equivalence \eqref{E:Knoerrer2} would imply $[d-1] \cong \mathbb{S} \cong \mathbb{S}_1  \cong [d]$, which would show $[1] \cong \id$ in $D_{sg}(R) \cong D_{sg}(R_1)$.
In particular, $X[1]\cong X$ for every indecomposable object $X$ in $D_{sg}(R)$. Identifying singularity categories with homotopy categories of matrix factorizations (via \cite{Eis80, Buchweitz}), it follows from \cite[Prop. 2.7. i) \& Lemma 2.5. ii)]{KnorrerCMmodules} that there is an indecomposable object $Y$ in $D_{sg}(R_1)$ such that $Y[1]\not\cong Y$. This contradicts $[1] \cong \id$ in $D_{sg}(R_1)$ and shows that \eqref{E:Knoerrer2} is impossible.


}


\begin{thebibliography}{17}
\mm{\bibitem{Aus78} M.~Auslander, \emph{Functors and morphisms determined by objects}, Representation
  theory of algebras (Proc. Conf., Temple Univ., Philadelphia, Pa., 1976),
  Dekker, New York, 1978, pp.~1--244. Lecture Notes in Pure Appl. Math., Vol.
  37.}
   \bibitem{Aus84}
\mm{\bysame, \emph{{Isolated singularities and existence of almost split
  sequences. Notes by Louise Unger.}}, {Representation theory II, Groups and
  orders, Proc. 4th Int. Conf., Ottawa/Can. 1984, Lect. Notes Math. 1178,
  194-242 (1986)}.}
\bibitem{AuslanderReiten} M.~Auslander, I.~Reiten, \emph{The Cohen-Macaulay type of Cohen-Macaulay rings}, Advances in Mathematics 73.1 (1989): 1-23.
 \bibitem{Bel00} A.~Beligiannis, \emph{The homological theory of contravariantly finite subcategories: Auslander-Buchweitz contexts, Gorenstein categories and (co-)stabilization},
Comm. in Algebra \textbf{28} (10), 4547--4596, 2000.
 \bibitem{BM} A.~Beligiannis, N.~Marmaridis, \emph{Left triangulated categories arising from contravariantly finite subcategories}, Comm. Algebra, \textbf{22}(12):5021--5036, 1994.
\mm{\bibitem{BondalKapranov} A.~Bondal, M.~Kapranov, \emph{Representable functors, Serre functors, and reconstructions}, Izv. Akad. NaukSSSR Ser. Mat. 53 (1989), no. 6, 1183--1205.}
\bibitem{Buchweitz} R.-O.~Buchweitz, \emph{Maximal Cohen-Macaulay modules and Tate-Cohomology over
Gorenstein rings}, Preprint 1987, available at \url{http://hdl.handle.net/1807/16682}.
\bibitem{ChenQuadratic}  X.-W.~Chen, \emph{The singularity category of a quadratic monomial algebra},
The Quarterly Journal of Mathematics, \textbf{69} (2018), 1015--1033. 
\bibitem{ChenYang}
X.-W.~Chen, D.~Yang,
\newblock \emph{Homotopy categories, {L}eavitt path algebras, and {G}orenstein
  projective modules},
\newblock {IMRN}, (10):2597--2633, 2015.
\mm{\bibitem{Eis80}
D.~Eisenbud, \emph{Homological algebra on a complete intersection, with an application to group representations}, Trans.\ Amer.\ Math.\ Soc.\ {\bf 260} (1980) 35--64.}
\bibitem{Heller} A.~Heller, \emph{Stable homotopy categories}, Bull. Amer. Math. Soc. \textbf{74} (1968), 28--63.
\mm{\bibitem{HIMO} M. Herschend, O. Iyama, H. Minamoto, S. Oppermann, \emph{Representation theory of Geigle--Lenzing complete intersections}, Mem. Amer. Math. Soc. (to appear), arXiv:1409.0668.}
\mm{\bibitem{IyamaWemyss} O. Iyama, M. Wemyss, \emph{Singular derived categories of $\mathbb{Q}$-factorial terminalizations and maximal modification algebras}, Adv. Math. 261 (2014), 85--121.}
 \bibitem{KIWY15} M.~Kalck, O.~Iyama, M.~Wemyss, D.~Yang, \emph{Frobenius categories, Gorenstein algebras and rational surface singularities},   
 Compositio Mathematica, vol. \textbf{151}, issue 03, 502--534 (2015).
\bibitem{KK17} M.~Kalck, J.~Karmazyn, \emph{Noncommutative Kn{\"o}rrer type equivalences via noncommutative resolutions of singularities},
arXiv:1707.02836, \href{https://arxiv.org/pdf/1707.02836.pdf}{(pdf)}.
\bibitem{KY16} M.~Kalck, D.~Yang, \emph{Relative singularity categories I: Auslander resolutions}, 
  Advances in Mathematics, vol. \textbf{301}, 973--1021 (2016), \href{https://arxiv.org/pdf/1205.1008.pdf}{(pdf)}.  
\bibitem{KY18} M.~Kalck, D.~Yang, \emph{Relative singularity categories II: DG models},   
  arXiv1803.08192, \href{https://arxiv.org/pdf/1803.08192.pdf}{(pdf)}. 
\mm{\bibitem{KashiwaraSchapira} M.~Kashiwara, P.~Schapira, \emph{Categories and Sheaves},
Grundlehren der Mathematischen Wissenschaften 332, Springer (2006).}
\bibitem{Kawamata15}
Y.~Kawamata,
\newblock \emph{On multi-pointed noncommutative deformations and Calabi-Yau threefolds}, Compos. Math. 154 (2018), no. 9, 1815--1842.
\bibitem{KellerVossieck87}
B.~Keller, D.~Vossieck, \emph{Sous les cat{\'e}gories
  d{\'e}riv{\'e}es}, C. R. Acad. Sci. Paris S{\'e}r. I Math. \textbf{305}
  (1987), no.~6, 225--228. 
  \bibitem{KR}
M.~Khovanov, L.~Rozansky,
\newblock \emph{Matrix factorizations and link homology},
\newblock {Fund. Math.}, 199(1):1--91, 2008.
\bibitem{KnorrerCMmodules}
H.~Kn{\"o}rrer, \emph{Cohen-{M}acaulay modules on hypersurface singularities {I}}, Invent. Math. 88(1):153--164, 1987.
\bibitem{Leavitt} W.~Leavitt, \emph{The module type of a ring}, Trans. Amer. Math. Soc. 103 (1962), 113--130.
\bibitem{Matsui} H.~Matsui, \emph{Singular equivalences of commutative noetherian rings and reconstruction of singular loci}, J. Algebra 522 (2019): 170--194.
\bibitem{Orlov2} D.~Orlov, 
\emph{Triangulated categories of singularities and {D}-branes in {L}andau-{G}inzburg models},
Tr. Mat. Inst. Steklova \textbf{246} (2004), Algebr. Geom. Metody, Svyazi i Prilozh., 240--262.
\bibitem{OrlovIdempotent}
\bysame,
\newblock \emph{Formal completions and idempotent completions of triangulated
  categories of singularities},
\newblock {Adv. Math.}, \textbf{226}(1):206--217, 2011.
\bibitem{PavicShinder} N.~Pavic, E.~Shinder, \emph{K-theory and the singularity category of quotient singularities}, arXiv:1809.10919.
\bibitem{Wemyss18}
M.~Wemyss, \emph{Flops and clusters in the homological minimal program},
  Invent. Math. 211 (2018), 435--521. 
\bibitem{YangPrivate} D.~Yang, Private Communication, Jan 2015.
\bibitem{Yoshino} Y. Yoshino, \emph{Cohen-Macaulay modules over Cohen-Macaulay rings}. London Mathematical Society Lecture Note Series, 146. Cambridge University Press, Cambridge, 1990.
\end{thebibliography}
\end{document}